\documentclass[a4paper, 11pt]{article}
\usepackage{graphics,graphicx}
\usepackage{amssymb,amsmath}
\usepackage{enumerate}
\usepackage{fullpage}

\usepackage{color}

\newtheorem{theorem}{Theorem}
\newtheorem{problem}{Problem}
\newtheorem{claim}{}[theorem]
\newtheorem{cor}{Corollary}[theorem]
\newtheorem{lemma}{Lemma}

\newtheorem{conjecture}{Conjecture}

\newtheorem{proposition}{Proposition}
\def \no {\noindent}
 
 \def \sm {\setminus}
 \def \es {\emptyset}

\newenvironment{proof}[1][]%
{\noindent {\setcounter{equation}{0}\it Proof.
}{#1}{}}{\hfill$\Box$\vspace{2ex}}

{\noindent {\setcounter{equation}{0}\it Proof.
}{#1}{}}{\hfill$\Diamond$\vspace{2ex}}

\title{Coloring  graph classes with no induced fork via perfect divisibility}
\author{T. Karthick\thanks{Computer Science Unit, Indian Statistical
Institute, Chennai Centre, Chennai 600029, India. This research is partially supported by DST-SERB, Government of India under MATRICS scheme (MTR/2018/000288).} ~~~~ Jenny Kaufmann\thanks{Harvard University, Cambridge, MA 02138. This work was performed when the author was at Princeton University.} ~~~~ Vaidy Sivaraman\thanks{Department of Mathematics and Statistics, Mississippi State University, Mississippi State, MS 39762, USA.}}

\begin{document}
\maketitle

\begin{abstract}
For a graph $G$, $\chi(G)$ will denote its chromatic number, and $\omega(G)$ its clique number.
A graph $G$ is said to be \emph{perfectly divisible} if for all induced subgraphs $H$ of $G$, $V(H)$ can
be partitioned into two sets $A$, $B$ such that $H[A]$ is perfect and $\omega(H[B]) < \omega(H)$. An integer-valued  function $f$ is called a {\it $\chi$-binding function} for a hereditary class of graphs $\cal C$ if $\chi(G) \leq f(\omega(G))$ for every graph $G\in \cal C$.
The \emph{fork} is the graph obtained from the complete bipartite graph $K_{1,3}$ by subdividing an edge once.
The problem of finding a polynomial $\chi$-binding function for the class of fork-free graphs is open.  In this paper, we study the structure of some classes of fork-free graphs; in particular, we study the class of (fork,\,$F$)-free graphs $\cal G$ in the context of perfect divisibility, where $F$ is a graph on five vertices with a stable set of size  three, and show that every $G\in \cal G$ satisfies $\chi(G)\leq \omega(G)^2$. We also note that the class $\cal G$ does not admit a linear $\chi$-binding function.
\end{abstract}

\medskip
\no{\bf Keywords}:  Fork-free graphs; Perfect divisibility; Chromatic number; Clique number.

\section{Introduction}
\underline{}For a positive integer $n$, $K_n$ will denote the complete graph on $n$ vertices, and $P_n$ will denote
the path on $n$ vertices. For an integer $n > 2$, $C_n$ will denote the cycle on $n$ vertices. A \emph{hole} in a graph is an induced cycle $C_n$ with $n > 3$; an \emph{antihole} is the complement of a hole. A hole or antihole is \emph{odd} (\emph{even}) if $n$ is odd (even).  The
\emph{union} of two vertex-disjoint graphs $G_1$ and $G_2$, denoted by $G_1\cup G_2$, is the graph
with vertex set $V(G_1)\cup V(G_2)$ and edge set $E(G_1)\cup E(G_2)$.  The
union of $k$ copies of the same graph $G$ will be denoted by $kG$.  A \emph{stable set} (or an \emph{independent set}) is a set of vertices that are pairwise nonadjacent.

A class of graphs $\cal C$ is \emph{hereditary} if every induced subgraph of every graph in $\cal C$ is also in $\cal C$.
An important and well studied type of hereditary class of graphs is the class of graphs which are defined by forbidden induced subgraphs.  Given a graph $H$, we say that a graph $G$ is \emph{$H$-free} if $G$ has no induced subgraph that is isomorphic to $H$. Given a class of graphs $\cal H$, we say that a graph $G$ is \emph{$\cal H$-free}  if $G$ is $H$-free for every $H\in \cal H$.

For a graph $G$, $\chi(G)$ will denote its chromatic number, and $\omega(G)$ its clique number.
 For every graph $G$, $\chi(G) \geq \omega(G)$. A graph $G$ is called \emph{perfect} if for every induced subgraph $H$ of $G$, $\chi(H) = \omega(H)$.

A graph $G$ is said to be \emph{perfectly divisible} if for all induced subgraphs $H$ of $G$, $V(H)$ can
be partitioned into two sets $A$, $B$ such that $H[A]$ is perfect and $\omega(H[B]) < \omega(H)$. Perfectly divisible graphs were introduced by Ho\'ang \cite{Hoang-2}, and can be thought of as a generalization of perfect graphs in the sense that perfect graphs are perfectly divisible. However not all perfectly divisible graphs are perfect. For example, the vertex set of  an odd hole  can be partitioned into two sets such that the first set induces a perfect graph and the other is a stable set. So an odd hole is  perfectly divisible but not perfect. Ho\'ang \cite{Hoang} observed that the class of $3K_1$-free graphs is perfectly divisible, and in \cite{Hoang-2} he showed that the class of (banner, odd hole)-free graphs is perfectly divisible. Chudnovsky and the third author  showed that the class of ($P_5$, bull)-free graphs is perfectly divisible \cite{CS2018}. See Sections~2 and 4 for more on perfect divisibility.

 A hereditary class of graphs $\cal C$ is \emph{$\chi$-bounded} if there is a function $f$ (called a {\it $\chi$-binding function}) such that $\chi(G) \leq f(\omega(G))$ for every graph $G\in \cal C$. In addition, if $f$ is a polynomial function  then the class $\cal C$ is \emph{polynomially $\chi$-bounded}. It has long been known that there are hereditary graph classes that are not $\chi$-bounded (see  \cite{SS20} for examples) but it is not known whether there is a hereditary graph class that is $\chi$-bounded but not polynomially $\chi$-bounded. A recent survey of Scott and Seymour \cite{SS20}  gives a detailed overview of this area of research.

The class of claw-free graphs (or $K_{1,3}$-free graphs) is widely studied in a variety of contexts and has a vast literature; see \cite{FFR} for a survey.  A detailed and complete structural classification of claw-free graphs has been given by Chudnovsky and Seymour; see \cite{CS}.   A result of Gy\'arf\'as \cite{AG} together with a result of Kim \cite{JHK} show that the class  of claw-free graphs is $\chi$-bounded,  and  that every such graph $G$ satisfies $\chi(G)\leq O(\omega(G)^2/\log \omega(G))$. It is also known that  there is no linear $\chi$-binding function even for a very  special class of  claw-free graphs; see \cite{Brause}. Chudnovsky and Seymour \cite{CSey2010} showed that every connected  claw-free graph $G$ with a stable set of size at least $3$ satisfies  $\chi(G) \leq 2\omega(G)$.

The \emph{fork} is the graph obtained from the complete bipartite graph $K_{1,3}$ by subdividing an edge once. The class of claw-free graphs is a subclass of the class of fork-free graphs. It is a natural line of research to see what properties of claw-free graphs are also enjoyed by fork-free graphs. A classic example is the polynomial-time solvability of the (weighted) stable set problem in the class of fork-free graphs \cite{VEA, LM}, generalizing the result for claw-free graphs \cite{GM,NS}. It has long been known that the class of fork-free graphs is $\chi$-bounded \cite{KP}, and it not known whether the class of fork-free graphs is polynomially $\chi$-bounded or not.   Indeed, Randerath and Schiermeyer \cite{RS} asked the following  interesting question.

\begin{problem}
Does there exist a polynomial $\chi$-binding function for the class of fork-free graphs?
\end{problem}

The third author (unpublished) has conjectured  that the class of fork-free graphs is perfectly divisible which in turn will yield a quadratic $\chi$-binding function. In this paper, we are interested in polynomial $\chi$-binding functions for some classes of fork-free graphs, namely (fork,\,$F$)-free graphs $\cal G$, where $F$ is any nontrivial graph on at most five vertices, and we give below some known results in this direction.

\begin{itemize}\itemsep=0pt

\item If $F=K_3$, then it is observed in \cite{RS} that every $G\in \cal G$ satisfies $\chi(G)\leq 3$. Moreover if $G$ is connected, then equality holds if and only if $G$ is an odd hole.

\item If $F\in \{P_3, \overline{P_3}\}$, then clearly every $G\in \cal G$ is perfect.

\item  If $F\in \{P_4,C_4,K_4,K_4-e, K_3\cup K_1, \mbox{paw} \}$, then  $\cal G$ is linearly $\chi$-bounded; see \cite{CHKK} and the reference therein.

\item  It follows from a result of Wagon \cite{wagon}  that, if $F=2K_2$, then  every $G\in \cal G$ satisfies $\chi(G)\le \binom{\omega(G)+1}{2}$. Further, it is known that $\cal G$ does not admit a linear $\chi$-binding function; see \cite{Brause}.

\item If $F=K_{1,3}$, then every $G\in \cal G$ satisfies $\chi(G)\le O(\omega(G)^2/\log \omega(G))$, and $\cal G$ does not admit a linear $\chi$-binding function; see \cite{AG,JHK}.

\item If $F\in \{P_3\cup K_1, K_2\cup 2K_1\}$, then it follows from Theorem~18 of \cite{RS-Survey} that $\cal G$ is quadratically $\chi$-bounded.

\item  Randerath  \cite{Rand}   showed that, if $F\in \{\overline{P_3\cup 2K_1},K_5-e\}$, then   every $G\in \cal G$ satisfies $\chi(G)\leq \omega(G)+1$.

\item Recently, Chudnovsky et al \cite{CCS} proved a structure theorem for the class of (fork, antifork)-free graphs, and used it to prove that every (fork, antifork)-free graph $G$ satisfies $\chi(G)\leq 2\omega(G)$. (Here, an  \emph{antifork} is the complement graph of a fork.)

\end{itemize}
Thus if $|V(F)|\leq 4$, then the class of (fork,\,$F$)-free graphs is known to be quadratically $\chi$-bounded except when $F=4K_1$, and not much is known when $|V(F)|=5$.    Here, we study the class of (fork,\,$F$)-free graphs, where $F$ is a graph on five vertices with a stable set of size  three. More precisely, we consider the class of (fork,\,$F$)-free graphs $\cal F$, where $F$ is one of the following graphs: $P_6$, dart, co-dart, co-cricket, banner, and bull, and show that the following hold:
\begin{itemize}\itemsep=0pt
\item[(i)] Every $G\in \cal F$ is   perfectly divisible, when $F\in \{P_6, \text{co-dart}, \text{bull}\}$.
\item[(ii)] Every $G\in \cal F$ is either claw-free or perfectly divisible, when $F\in$ \{dart, banner, co-cricket\}.
\item[(iii)] Every $G\in \cal F$ satisfies $\chi(G)\leq \omega(G)^2$.
\item[(iv)] Since the class of $3K_1$-free graphs does not admit a linear $\chi$-binding function \cite{Brause}, and since each graph  $F$ and the fork  has a stable set of size $3$,  it follows that the class  $\cal F$   does not admit a linear $\chi$-binding function.
\end{itemize}

\section{Preliminaries}
We follow West \cite{West} for standard notation and terminology used here, and we refer to the website `\emph{https://www.graphclasses.org/smallgraphs.html}' for some special graphs used in this paper.
For a vertex $v$ in a graph $G$, $N_G(v)$ is the set of vertices adjacent to $v$,  $N_G[v]$ is the set $\{v\} \cup N_G(v)$, and $M_G(v)$ is the set $V(G) \setminus N_G[v]$.   Given a subset $X \subseteq V(G)$,     $N_G(X)$   is the set $\{u \in
V(G)\setminus X : u$ $ \mbox{~is adjacent to a vertex of }X\}$, and $M_G(X)$ is the set $V(G)\setminus (X\cup N_G(X))$.    We drop the subscript $G$ in the above notations if there is no ambiguity.
  For a vertex set $X \subseteq V(G)$, $G[X]$ denotes the subgraph of $G$ induced by $X$. We say that a graph $G$ contains a graph $H$ if $H$ is an induced subgraph of $G$.  The complement of a graph $G$ will be denoted $\overline{G}$.
  Given disjoint vertex sets $S, T$, we say that $S$ is \emph{complete} to $T$ if every vertex in $S$ is adjacent to
every vertex in $T$; we say $S$ is \emph{anticomplete} to $T$ if every vertex in $S$ is nonadjacent to every vertex
in $T$; and we say $S$ is \emph{mixed} on $T$ if $S$ is not complete or anticomplete to $T$. When $S$ has a
single vertex, say $v$, we can instead say that $v$ is complete to, anticomplete to, or mixed on $T$.  A vertex $v$ in $G$ is \emph{universal} if it is complete to $V(G)\setminus \{v\}$. A set $S \subseteq V(G)$ is a \emph{homogeneous set} if $1 < |S| < |V(G)|$ and
for every   $v\in V(G)\setminus S$, $v$  is either complete or anticomplete to $S$. We say that a graph $G$ admits a homogeneous set decomposition  if $G$ has a homogeneous set.
  The
\emph{independence number} $\alpha(G)$ of a graph $G$ is the size of a largest stable set in $G$. A \emph{triad} in a graph $G$ is a stable set of size $3$.

For  a vertex subset $S:=\{v_1, v_2,\ldots,v_k\}$ of $G$, we write $v_1$-$v_2$-$\cdots$-$v_k$-$v_1$ to denote  an induced cycle $C_k$ in $G$ with vertex set $S$ and edge set $\{v_1v_2,v_2v_3, \ldots,$ $v_{k-1}v_k,v_kv_1\}$, and we write $v_1$-$v_2$-$\cdots$-$v_k$ to denote an induced path $P_k$ in $G$ with vertex set $S$ and edge set $\{v_1v_2,v_2v_3, \ldots,v_{k-1}v_k\}$.

 We say that a vertex $v$ is a
\emph{center}  of a claw  in a graph $G$, if $v$ has neighbors $a,b, c \in V(G)$ such that $\{a, b, c\}$ is a triad; and we call the vertices $a$, $b$, $c$   the \emph{leaves} of the claw.

The \emph{paw} is the graph that consists of a $K_3$ with a pendant vertex attached to it. The \emph{diamond} is the graph $K_4-e$.

  We say that a vertex set $\{v_1,v_2,v_3,v_4,v_5\}$ induces:
  \begin{itemize}\itemsep=0pt
  \item a \emph{fork} if $\{v_2,v_3,v_4,v_5\}$ induces a claw with center $v_3$, and $v_1$ is a leaf adjacent to $v_2$.

\item a \emph{dart} if $\{v_1, v_2, v_3, v_4\}$ induces a claw with center $v_1$, and $v_5$ is adjacent to $v_1,v_3,$ and $v_4$ but not to $v_2$.

\item a \emph{banner} if $\{v_1, v_2, v_3, v_4\}$ induces a claw with center $v_1$, and $v_5$ is adjacent to $v_2,v_3$ but not to $v_1,v_4$.

\item a  \emph{co-dart} if $\{v_1,v_2,v_3,v_4\}$ induces a paw, and $v_5$ is anticomplete to $\{v_1,v_2,$ $v_3,v_4\}$.
\item a  \emph{co-cricket} if $\{v_1,v_2,v_3,v_4\}$ induces a diamond, and $v_5$ is anticomplete to $\{v_1,v_2,v_3,v_4\}$.

\item a \emph{bull} if $v_1$-$v_2$-$v_3$-$v_4$ is a path, and $v_5$ is adjacent to $v_2,v_3$ but not to $v_1,v_4$.
\end{itemize}

 Note that a co-cricket $\cong \text{diamond}\cup K_1$, and  a co-dart $\cong \text{paw}\cup K_1$ and is the complement graph of a dart.

\medskip

We use the following known results.  The class of perfect graphs admits a forbidden induced subgraph characterization, namely, \emph{the strong perfect graph theorem},   given below.

\begin{theorem}[\cite{spgt}]\label{thm:spgt}
A graph is perfect if and only if it has no odd hole or odd antihole.
\end{theorem}

  Chudnovsky and Seymour \cite{CSey2010}   give a simple proof of the following $\chi$-bound for claw-free graphs
in general.

\begin{theorem}[\cite{CSey2010}]\label{thm:claw-free}
 Every claw-free graph $G$ satisfies $\chi(G)\leq \omega(G)^2$. Moreover, the bound is asymptotically tight.
\end{theorem}

 Although perfect divisibility is a structural property, it immediately implies a quadratic $\chi$-binding function. Indeed, we have the following (see also \cite{CS2018}).

\begin{lemma}[\cite{Hoang-2}]\label{lem:pd-bound}
 Every perfectly divisible graph $G$ satisfies $\chi(G) \le \binom{\omega(G)+1}{2}$.
\end{lemma}

 A graph $G$ is said to be \emph{perfectly  weight  divisible} if for every nonnegative integer weight function $w$ on $V(G)$, there is a partition of $V(G)$ into two sets $S$ and  $T$ such that $G[S]$ is perfect and the maximum weight of a clique in $G[T]$ is smaller than the maximum weight of a clique in $G$.
 We will also use the following results.

\begin{theorem}[\cite{CS2018}]\label{thm:hsd}
A minimal non-perfectly weight divisible graph does not admit a homogeneous set
decomposition.
\end{theorem}

The proof of the following theorem  is similar to the proof of Theorem~3.7 in \cite{CS2018}, and we give it here for completeness.

\begin{theorem}\label{pwd=pd}
Let $\cal C$ be a hereditary class of graphs. Suppose that every graph $H \in \cal C$ has a vertex $v$ such that $H[M_H(v)]$ is perfect.  Then every $G\in \cal C$ is perfectly weight divisible, and hence  perfectly divisible.
\end{theorem}
\begin{proof} Let $G\in \cal C$ be a minimal counterexample to the theorem. Then there is a nonnegative integer weight function $w$ on $V(G)$ for which there is no partition of $V(G)$ as in the definition of perfectly weight divisibility. Let $U$ be the set  $\{v\in V(G)\mid w(v)>0\}$, and let $H\cong G[U]$. Since $\cal C$ is hereditary, $H\in \cal C$, and so by the  hypothesis of the theorem, $H$ has a vertex $v$ such that $H[M_{H}(v)]$ is perfect. But now, since $w(v)>0$, if we let $S=M_{H}(v)\cup \{v\}$ and $T=N_{H}(v)\cup (V(G)\setminus U)$, then we get a partition of $V(G)$ as in the definition of   perfectly weight divisibility, a contradiction. This proves the theorem.
\end{proof}

\section{Classes of  fork-free graphs}

\subsection{The class of (fork,\,$P_6$)-free graphs}

In this section we prove that (fork,\,$P_6$)-free graphs are perfectly divisible, and hence the class of (fork,\,$P_6$)-free graphs is quadratically $\chi$-bounded. A vertex set $X$ in a graph $G$ is said to be \emph{anticonnected} if the subgraph induced by $X$ in $\overline{G}$ is connected. Also, a vertex $v$ is an \emph{anticenter} for a vertex set $X$ if $N[v] \cap X = \emptyset$. An \emph{antipath} in $G$ is the complement of the path $v_1$-$v_2$-$\cdots$-$v_k$ in $G$, for some $k$.

\begin{theorem}\label{FP6freeStruc}
Let $G$ be a (fork,\,$P_6$)-free graph which is not perfectly divisible. Then $G$ has a homogeneous set.
\end{theorem}

\begin{proof}
Since $G$ is not perfectly divisible, given  $v\in V(G)$,    $G[M(v)]$ is not perfect, so it contains an odd hole $C_n$ or an odd antihole $\overline{C_n}$, by Theorem~\ref{thm:spgt}. Note that $C_5 \cong \overline{C_5}$, and since $G$ is $P_6$-free, $G$ does not contain $C_n$ for $n > 6$. So $G$ contains an odd antihole induced by $X_0$ with anticenter $v$.
We construct a sequence of vertex sets $X_0,X_1,\ldots,X_t,\ldots$ such that for each $i$, $X_i$ is obtained from $X_{i-1}$ by adding one vertex $x_i$ that has a neighbor and a nonneighbor in $X_{i-1}$. Let $X$ be the maximal vertex set obtained this way. By maximality, $X$ is a  homogeneous set. We show that $X \neq V(G)$; in particular, we  show that $X$ does not intersect the set $A$ of anticenters for $X_0$.

\begin{claim}\label{NAX0}
$N(A)$ is complete to $X_0$.
\end{claim}
 Let $v_1, v_2, \ldots, v_n$ be the vertices of $X_0$, with edges $v_iv_j$ whenever $|i - j| \neq 1$ (indices are modulo $n$). Suppose to the contrary that there exists a vertex $b \in N(A)$ that has a nonneighbor in $X_0$.   By assumption,  $b$ has a neighbor in $X_0$, and a neighbor $a \in A$.
 Suppose that $b$ has two consecutive neighbors in $X_0$.  Then  there is some $i$ such that $v_i, v_{i+1} \in N(b)$ and $v_{i+2} \not\in N(b)$. But then  $\{v_{i+2},v_i,b,v_{i+1},a\}$ induces a fork. So we may assume that $b$ does not have two consecutive neighbors in $X_0$. Then since $n$ is odd, there must exist some $i$ such that $v_i,v_{i+1},v_{i+3}\notin N(b)$ and $v_{i+2} \in N(b)$. But now $a$-$b$-$v_{i+2}$-$v_i$-$v_{i+3}$-$v_{i+1}$ is a $P_6$ which is a contradiction. This proves \ref{NAX0}.

\begin{claim}\label{ANAXi}
For each $i\in \{0,1,2,\ldots,\}$, $X_i$ is complete to $N(A)$, and $X_i$ does not intersect $A \cup N(A)$.
\end{claim}
 We prove the assertion by induction on $i$. By \ref{NAX0}, we may assume that $i\geq 1$. Suppose to the contrary that $x_i$ has a nonneighbor $b \in N(A)$; then $b$ is a center for $X_{i-1}$. Since $x_i$ has a neighbor in $X_{i-1}$ and $N(A) \cap X_{i-1}=\es$, we have $x_i \not\in A$. Since $x_i$ has a nonneighbor in $X_{i-1}$ and $N(A)$ is complete to $X_{i-1}$, we have $x_i \not\in N(A)$.  Since $x_i \not\in A$, $x_i$ has a neighbor in $X_0$, say $v_1$. Note that $X_i$ is anticonnected for each $i$. Then if $\overline{P}$ is a shortest antipath from $x_i$ to $v_1$, $\overline{P}$ contains at least three vertices; label its vertices $x_i=w_1$-$w_2$-$\cdots$-$w_t=v_1$ in order. Since $w_2, w_3 \in X_{i-1} \subseteq N(b)$, $\{x_i,w_3,b,w_2,a\}$ induce a fork which is a contradiction. So $x_i$ has no nonneighbors in $N(A)$. This proves \ref{ANAXi}.

\medskip
Now by \ref{ANAXi}, it follows  that the vertex set $X$ does not intersect $A \cup N(A)$. Since $A\neq \es$, we have $X \neq V(G)$, as desired.
\end{proof}

\begin{cor}\label{FP6-pd}
Every (fork,\,$P_6$)-free graph is perfectly divisible.
\end{cor}
\begin{proof}This follows from Theorems~\ref{thm:hsd} and \ref{FP6freeStruc}. \end{proof}

We immediately have the following corollary which generalizes the result that the class of $3K_1$-free graphs is perfectly divisible \cite{Hoang}.

\begin{cor}\label{P3uK1-pd}
Every $P_3\cup K_1$-free graph is perfectly divisible.
\end{cor}

\begin{cor}\label{FP6-bnd}
Every (fork,\,$P_6$)-free graph $G$ satisfies $\chi(G)\le \binom{\omega(G)+1}{2}$.
\end{cor}
\begin{proof}This follows from Lemma~\ref{lem:pd-bound}. \end{proof}

\subsection{The class of (fork, dart)-free graphs}

In this section, we prove that the class of (fork, dart)-free graphs is quadratically $\chi$-bounded.

\begin{theorem}\label{FDfreeStruc-1}
Let $G$ be a connected (fork,\,dart)-free graph. If $G$ contains a claw with center $v$, then $G[M(v)]$ is perfect.
\end{theorem}
\begin{proof}
Let $G$ be a (fork,\,dart)-free graph containing a claw with center $v$. Let $L$ be the set of leaves of claws in $G$ with center $v$. Then $L$ has a triad, and so $|L|\geq 3$. Let $Y$ denote the set $M(v) \cap N(L)$, and  $X$ denote the set $M(v) \setminus N(L)$. Then $M(v)$ is the set $X \cup Y$.

\begin{claim}\label{YL}
If $y \in Y$ has a neighbor in a triad $T$ in $L$, then it has at least two neighbors in $T$. Moreover, every $y \in Y$ has a neighbor in a triad $T$, so has two nonadjacent neighbors in $T$.
\end{claim}
 Let $\{a,b,c\}$ be a triad in $L$, and suppose that $y$ is adjacent to $a$. Then since $\{y,a,v,b,c\}$ does not induce a fork, we see that $y$ is adjacent to $b$ or $c$. This prove the first assertion of \ref{YL}. Note that by definition, any $y$ in $Y$ has a neighbor $a \in L$; then there exist $b,c\in L$ such that $\{a,b,c\}$ is a triad, and hence $y$ is also adjacent to $b$ or $c$. This proves \ref{YL}.

\begin{claim}\label{NeiP3free}
If $a \in L$, then $G[N(a)\cap Y]$ is $P_3$-free.
\end{claim}
If there is a $P_3$, say $y_1$-$y_2$-$y_3$, in $G[N(a)\cap Y]$, then $\{a,v,y_1,y_3,y_2\}$ induces a dart, a contradiction.  This proves \ref{NeiP3free}.

\begin{claim}\label{L2L}
$N(v) \setminus L$ is complete to $L$.
\end{claim}
Suppose to the contrary that $t\in N(v)\setminus L$ has a nonneighbor $a \in L$. Then since $|L|\geq 3$, there are vertices  $b$ and $c$ in $L$ such that $\{v, a, b, c\}$ is a claw. Now, if $t$ is not adjacent to $b$, then $t$ is a leaf of the claw induced by $\{v, a, b, t\}$, a contradiction to our assumption that $t\notin L$. So $t$ is adjacent to $b$. Likewise, $t$ is adjacent to $c$. But then $\{v,a,b,c,t\}$ induces a dart which is a contradiction. This proves \ref{L2L}.

\begin{claim}\label{NvX}
$N[v]$ is anticomplete to $X$.
\end{claim}
Suppose to the contrary that there exists a vertex $t\in N(v)$ such that $t$ has a  neighbor, say $x\in X$. Since  $L$ is anticomplete to $X$ (by the definition), we may assume that $t\in N(v)\sm L$. By \ref{L2L}, $\{t\}$ is complete to $L$. Then since $|L|\geq 3$, there exist $a,b\in L$ such that $\{t,x,a,b,v\}$ induces a dart which is a contradiction.  This proves \ref{NvX}.

\begin{claim}\label{yCompX}
Let $y\in Y$ and let $X'$ be a component of $X$. Then $y$ is not mixed on $V(X')$.
\end{claim}
Suppose not. We may assume that $y$ is mixed on an edge $xx'$ in $X'$. Then since $y\in Y$, by \ref{YL}, there exist nonadjacent vertices $a,b\in L$ such that $a,b\in N(y)$. But then $\{x', x, y, a, b\}$ induces a fork, a contradiction. This proves \ref{yCompX}.

\begin{claim}\label{XY}
$X$ is complete to $Y$.
\end{claim}
 By \ref{NvX}, $N(X) \subseteq Y$. Since $G$ is connected, it follows  from \ref{yCompX} that  every vertex in $X$ has a neighbor in $Y$.
Suppose to the contrary that $x \in X$ is mixed on $Y$. Let $y\in N(x)\cap Y$ and let $y'\notin N(x)\cap Y$. By \ref{YL}, let $a$ and $b$ be the neighbors of $y$ in $L$. Recall that, by \ref{NvX}, $x$ is anticomplete to $L$. Suppose that $y'$ is adjacent to $y$. Then since  $\{v,a,y,y',x\}$ does not induce a fork, $y'$ is adjacent to $a$. Likewise, $y'$ is adjacent to $b$. But then $\{y,x,a,b,y'\}$ induces a dart, a contradiction. So we may assume that $y'$ is not adjacent to $y$.
Then since $\{x,y,a,y',v\}$ does not  induce a fork, $y'$ is not adjacent to $a$. Likewise, $y'$ is not adjacent to $b$. Then by \ref{YL}, $y'$ has nonadjacent neighbors $a',b'\in L\setminus \{a,b\}$. Then since $\{y',a',v,a,b\}$ does not induce a fork or a dart, we may assume that $a'$ is adjacent to $a$, but not to $b$. Then since $\{a',v,y,y',a\}$ does not induce a dart, $y$ is not adjacent to $a'$. Now,  $\{a',a,y,x,b\}$ induces a fork which is a contradiction. This proves \ref{XY}.

\begin{claim}\label{clX}
$X$ is a clique.
\end{claim}
 Suppose not. Let $x$ and $x'$ be two nonadjacent vertices in $X$.   By \ref{NvX}, $N(X) \subseteq Y$.  Now choose any $y \in Y$ and any $a \in L$ adjacent to $y$. Then, by \ref{NvX} and \ref{XY}, $\{v,a,y,x,x'\}$ induces a fork which is a contradiction. This proves \ref{clX}.

\begin{claim}\label{holeM}
If $C$ is an odd hole or  an odd antihole in $G[M(v)]$, then $V(C)\subseteq Y$.
\end{claim}
By \ref{XY} and \ref{clX}, every vertex in $X$ is  universal in $G[X \cup Y]$. Since odd holes and odd antiholes have no universal vertices, we see that  $V(C)\cap X=\es$. So $V(C)\subseteq Y$. This proves \ref{holeM}.

\begin{claim}\label{LoH}
  Let  $C:=$ $y_1$-$y_2$-$\ldots$-$y_n$-$y_1$ be an odd hole in $G[Y]$. Then every vertex in $L$ which has a neighbor in $C$ is adjacent to exactly two consecutive vertices of $C$.
  \end{claim}
Let $a\in L$.  We may assume that $y_1$ is a neighbor of $a$ in $C$. Then since $\{v,a,y_1,y_2,y_n\}$ does not induce a fork or a dart, we may assume that $a$ is adjacent to $y_2$, and is nonadjacent to $y_n$. Then since $\{a,v,y_1,y_3,y_2\}$ does not induce a dart, $a$ is not adjacent to $y_3$.  If $a$ has a neighbor in $C\setminus\{y_1,y_2\}$, say $y_i$   with the largest index $i$, then $3<i<n$. By the choice of $i$, we have $y_{i+1}\notin N(a)$, and then $\{y_{i+1},y_i,a,y_2,v\}$ induces a fork. Thus $a$ is anticomplete to $C\setminus\{y_1, y_2\}$. Hence every vertex in $L$ which has a neighbor in $C$ is adjacent to exactly two consecutive vertices of $C$. This proves \ref{LoH}.

\begin{claim}\label{MOHfree}
$G[M(v)]$ is $C_{2k+1}$-free, where $k\geq 2$.
\end{claim}
Suppose to the contrary that $G[M(v)]$ contains an odd hole, say  $C:=$ $y_1$-$y_2$-$\ldots$-$y_{2k+1}$-$y_1$.  By \ref{holeM}, $V(C)\subseteq Y$. By \ref{YL}, let $\{a,b,c\}$ be a triad in $L$, and let $a$ and $b$ be the neighbors of $y_1$ in $L$. Then by \ref{LoH}, $\{a,b\}$ is anticomplete to $y_3$, and we may assume that $N(a)\cap V(C) = \{y_1,y_2\}$. If $b$ is adjacent to $y_2$,  then $\{y_2,y_3,a,b,y_1\}$ induces a dart. So we may assume that $b$ is not adjacent to $y_2$. Then by \ref{YL}, $c$ is a neighbor of $y_2$. As earlier, we see that $c$ is not adjacent to $y_1$, and hence by \ref{LoH}, $c$ is adjacent to $y_3$. But now $\{y_3,c,v,a,b\}$ induces a fork which is a contradiction.  This proves \ref{MOHfree}.

\begin{claim}\label{MCOHfree}
$G[M(v)]$ is   $\overline{C_{2k+1}}$-free, where $k\geq 3$.
\end{claim}
Suppose to the contrary that $G[M(v)]$ contains a  $\overline{C_{2k+1}}$, say $C$ with vertices $y_1,y_2,\ldots,y_{2k+1}$ and edges $y_iy_j$ whenever $|i-j|\neq 1$ (indices are modulo $2k+1$).  By \ref{holeM}, $V(C)\subseteq Y$. Let $a \in L$ be a neighbor of $y_2$. Consider any consecutive pair of vertices $y_i, y_{i+1} \in C\setminus \{y_1, y_2, y_3\}$. Then since $\{v,a,y_2,y_i,y_{i+1}\}$ does not induce a fork or a dart, $a$ is adjacent to exactly one of $y_i$, $y_{i+1}$.
   Therefore, if $a$ is adjacent to $y_4$, then $y$ is adjacent to precisely the vertices with even index $i > 3$ in $C$, and if  $a$ is not adjacent to $y_4$, then $y$ is adjacent to precisely the vertices with odd index $i > 3$ in $C$. We may assume that $a$ is adjacent to $y_4$. Then $a$ is not adjacent to $y_{2k+1}$. Then since $y_2$-$y_4$-$y_1$ is a $P_3$, by \ref{NeiP3free}, $a$ is not adjacent to $y_1$.  But then $\{v,a,y_4,y_1,y_{2k+1}\}$ induces a fork which is a contradiction. This proves \ref{MCOHfree}.

\medskip
Now by \ref{MOHfree} and \ref{MCOHfree}, and by Theorem~\ref{thm:spgt}, we conclude that $G[M(v)]$ is perfect. This completes the proof. \end{proof}

\begin{theorem}\label{FDfreeStruc-2}
Let $G$ be a connected (fork,\,dart)-free graph. Then $G$ is either claw-free or for any claw in $G$ with center, say $v$, $G[M(v)]$ is perfect.
\end{theorem}

\begin{proof} This follows from Theorem~\ref{FDfreeStruc-1}.
\end{proof}

The following corollary generalizes the result known for the class of claw-free graphs (Theorem~\ref{thm:claw-free}).

\begin{cor}
Every (fork,\,dart)-free graph $G$ satisfies $\chi(G)\le \omega(G)^2$.
\end{cor}
\begin{proof}
Let $G$ be a (fork,\,dart)-free graph. We may assume that $G$ is connected.
If $G$ is claw-free, then the desired result follows from Theorem~\ref{thm:claw-free}. So let us assume that $G$ contains a claw with center, say $v$. Then by Theorem~\ref{FDfreeStruc-2},  $G[M(v)]$ is perfect. Now since $\omega(G[N(v)])\le \omega(G)-1<\omega(G)$ and since $G[\{v\}\cup M(v)]$ is perfect, we see that $G$ is perfectly divisible, and hence the result follows from Lemma~\ref{lem:pd-bound}.
\end{proof}

\subsection{The class of (fork, co-dart)-free graphs}

In this section, we prove that (fork, co-dart)-free graphs are perfectly divisible, and hence the class of (fork, co-dart)-free graphs is quadratically $\chi$-bounded.

\begin{theorem}\label{thm:fork-pawK1-free}
Let $G$ be a connected (fork, co-dart)-free graph.  Then either $G$ admits a homogeneous set decomposition or
for each vertex $v$ in $G$, $G[M(v)]$ is perfect.
\end{theorem}
\begin{proof}
Suppose to the contrary that $G$ does not admit a homogeneous set decomposition and that there is  a vertex $v$ in $G$ such that $G[M(v)]$ is not perfect. So by  Theorem~\ref{thm:spgt}, $G[M(v)]$ contains an odd hole or an odd antihole. Since $G$ has no co-dart, $G[M(v)]$ is paw-free, and so $G[M(v)]$ has no odd antiholes  except $\overline{C_5}$. So suppose that $G[M(v)]$ contains an odd hole.  Let $C:=$ $v_1$-$v_2$-$\cdots$-$v_{\ell}$-$v_1$ be a shortest odd hole in $G[M(v)]$ for some $\ell\ge 5$ with vertex set $S:=\{v_1,v_2,\ldots, v_\ell\}$.

%For $i\in \{1,2,\ldots,\ell\}$, let $A_i$ denote the set $\{x \in V(G) \sm S \mid |N(x)\cap S| = i$\}.
%Let $A_i^+$ denote the set  $\{x \in A_i \mid N(x) \cap Q_v \neq \emptyset\}$,
%and let $A^+:= \bigcup_{i = 1}^{\ell} A_i^+$. Likewise, let $A_i^-$ denote the set $\{x \in A_i \mid N(x) \cap Q_v =\emptyset\}$,
%and let $A^-:= \bigcup_{i = 1}^{\ell} A_i^-$. Let $Q_v$ denote the connected component of $G[V(G) \sm (N(S) \cup S)]$ containing $v$.

\begin{claim}\label{A-empty}
If $x\in N(S)$, then $x\in N(v)$.
\end{claim}
Suppose to the contrary that    $x$ is nonadjacent to $v$. First suppose that $x$ has two adjacent neighbours  in $C$. We may assume that $v_1,v_2\in N(x)$. Then since $\{v_1,v_2,v_3,x,v\}$ does not induce a co-dart, we see that $x$ is adjacent to $v_3$. Then by similar arguments, we conclude that $N(x)\cap S = S$. But now $\{v_1,v_2,v_{\ell-1},x,v\}$  induces a co-dart, a contradiction. So suppose that $x$ is nonadjacent to any two consecutive vertices in $C$. Since $x$ has a neighbor in $C$, we may assume that $x$ is adjacent to $v_1$. Then $x$ is nonadjacent to both $v_2$ and $v_{\ell}$. Then since $\{v_{\ell},v_1,v_2,v_3,x\}$ and  $\{v_2,v_1, v_{\ell},v_{\ell-1},x\}$ do not induce forks, $x$ is adjacent to both $v_3$ and $v_{\ell-1}$. Since $x$ is nonadjacent to   two consecutive vertices in $C$, this implies that $\ell\ge 7$, and $x$ is nonadjacent to $v_4$. But now $\{v_{{\ell}-1},x,v_3,v_4,v_2\}$ induces a fork, a contradiction. This proves \ref{A-empty}.

%\begin{claim}\label{xinA+}
%Let  $x$ be a vertex in $N(S)$. If $x$ is adjacent to $v_i$, for some $i\in \{1,2,\ldots, \ell\}$, then $x$ is adjacent to either $v_{i-1}$ or $v_{i+1}$.
%\end{claim}
%  Suppose that the assertion is not true.  By \ref{A-empty}, $x$ is adjacent to $v$. Now $\{v,x,v_i,v_{i+1},v_{i-1}\}$ induces a fork, a contradiction. This proves \ref{xinA+}.

\begin{claim}\label{A+=Al+}
Any vertex in $N(S)$ is complete to $S$.
\end{claim}
 Let  $x$ be a vertex in $N(S)$. Then by \ref{A-empty}, $x$ is adjacent to $v$. We may assume that $x$ is adjacent to $v_1$. Then since $\{v,x,v_1,v_2,v_{\ell}\}$ does not induce a fork, $x$ is adjacent to either $v_2$ or $v_{\ell}$. We may assume that $x$ is adjacent to $v_2$. Then for $j\in \{4,5,\ldots,\ell-1\}$, since $\{v_1,x,v,v_2,v_j\}$ does not induce a co-dart,  $x$ is complete to $\{v_4,v_5,\ldots, v_{\ell-1}\}$. Now suppose to the contrary that $x$ is nonadjacent to one of $v_3$ or $v_{\ell}$, say $v_{\ell}$. Then $x$ is adjacent to $v_3$. For, otherwise if $\ell =5$, then $\{v,x,v_4,v_{\ell},v_3\}$ induces a fork, and if $\ell \ge 7$, then $\{v_5,v_6,x,v,v_3\}$ induces a co-dart which are contradictions. But then $\{v,x,v_2,v_3,v_{\ell}\}$ induces a co-dart, a contradiction. So $x$ is  adjacent to both $v_3$ and $v_{\ell}$, and hence $x$ is complete to $S$.  This proves \ref{A+=Al+}.

\medskip
By \ref{A+=Al+}, we see that $S$ is a homogenous set, a contradiction. This proves Theorem~\ref{thm:fork-pawK1-free}.
\end{proof}

\begin{cor} \label{FCDfreeStruc-2}Every (fork,\,co-dart)-free graph is perfectly
weight divisible, and hence perfectly divisible.\end{cor}
\begin{proof}
Let $G$ be a minimal counterexample to the theorem. Then, by Theorem~\ref{thm:hsd}, $G$ does not admit a homogeneous set decomposition. So, by Theorem~\ref{thm:fork-pawK1-free}, there is a vertex $v$ in $G$ such that $G[M(v)]$ is perfect. Then,  by Theorem~\ref{pwd=pd}, it follows that $G$ is perfectly weight divisible, and hence perfectly divisible, a contradiction. This proves Corollary~\ref{FCDfreeStruc-2}. \end{proof}

\begin{cor}
Every (fork,\,co-dart)-free graph $G$ satisfies $\chi(G)\le \binom{\omega(G)+1}{2}$.
\end{cor}
\begin{proof}This follows from Corollary~\ref{FCDfreeStruc-2}, and from Lemma~\ref{lem:pd-bound}. \end{proof}

\subsection{The class of (fork, banner)-free graphs}

In this section, we prove that (fork, banner)-free graphs are either claw-free or perfectly
 divisible, and hence the class of (fork, banner)-free graphs is quadratically $\chi$-bounded. We use the following lemma.

\begin{lemma}[\cite{BLM}]\label{lem:B-free}
If $G$ is a banner-free graph that does not admit a homogeneous set decomposition, then $G$ is $K_{2,3}$-free.
\end{lemma}

\begin{theorem}\label{FBfreeStruc}
Let $G$ be a (fork,\,banner)-free graph that contains a claw. Then either $G$ admits a homogeneous set decomposition or there is a vertex $v$ in $G$ such that $G[M(v)]$ is perfect.
\end{theorem}
\begin{proof}
Let $G$ be a (fork,\,banner)-free graph that contains a claw.  Suppose that $G$ does not admit a homogeneous set decomposition. Then $G$ is connected, and, by Lemma~\ref{lem:B-free}, we may  assume that $G$ is $K_{2,3}$-free. Let $v$ be a vertex in $G$ such that $\alpha(G[N(v)])$ is maximized. Let $L$ be a maximum stable set in $N(v)$, and let $Q$ denote the set $N(v)\setminus L$. Since $G$ contains a claw, we see that $|L|\geq 3$ and so $L$ has a triad.

\begin{claim}\label{L-NA-M}
$M(v)$ is anticomplete to $L$.
\end{claim}
Suppose $x\in M(v)$ has a neighbor $a$ in a triad $\{a,b,c\}\subseteq L$.
Then since $\{v,a,b,c,x\}$ does not induce  an $K_{2,3}$ or a banner, $x$ is not adjacent to $b$ and $c$. But then  $\{v,a,b,c,x\}$ induces a fork, a contradiction. This proves \ref{L-NA-M}.

\begin{claim}\label{M-Cliques}
$G[M(v)]$ is a stable set.
\end{claim}
Suppose to the contrary that $G[M(v)]$ has a   component, say $C$ with more than one vertex. Then, by \ref{L-NA-M}, $V(C)$ is anticomplete to $L$.  Let $x,y\in V(C)$ be neighbors, and suppose $t\in Q$ is adjacent to $x$. Then $t$ is adjacent to at least two vertices in any given triad $\{a,b,c\}\subseteq L$ (otherwise, $G[\{x,t,v,a,b,c\}]$ contains a fork, a contradiction). We may assume $a,b\in N(t)$. Then since $\{y,x,t,a,b\}$ does not  induce a fork, $t$ is adjacent to $y$. Thus we conclude that every vertex in $N[v]$ is either complete or anticomplete to $V(C)$, and so $V(C)$ is a homogeneous set, a contradiction to our assumption. This proves \ref{M-Cliques}.

\medskip
Now it follows from \ref{M-Cliques} that $G[M(v)]$ is perfect.  This completes the proof.
\end{proof}

\begin{cor}\label{FBfreeStruc-2}
Let $G$ be a   (fork,\,banner)-free graph.  Then  either  $G$ is claw-free or $G$ admits a homogeneous set decomposition   or there is a vertex $v$ in $G$ such that $G[M(v)]$ is perfect.
\end{cor}
\begin{proof}This follows from Theorem~\ref{FBfreeStruc}. \end{proof}

\begin{cor} \label{FBfreeStruc-3}Let $G$ be a (fork,\,banner)-free graph.  Then either $G$ is claw-free or $G$ is perfectly
weight divisible, and hence perfectly divisible.\end{cor}
\begin{proof}This follows from Theorems~\ref{thm:hsd} and \ref{pwd=pd}, and from Corollary~\ref{FBfreeStruc-2}. \end{proof}

\begin{cor}
Every (fork,\,banner)-free graph $G$ satisfies $\chi(G)\le \omega(G)^2$.
\end{cor}
\begin{proof}This follows from  Corollary~\ref{FBfreeStruc-3}, Theorem~\ref{thm:claw-free}, and from Lemma~\ref{lem:pd-bound}. \end{proof}

\subsection{The class of (fork, co-cricket)-free graphs}

In this section, we prove that (fork, co-cricket)-free graphs are either claw-free or  perfectly
 divisible, and hence  the class of (fork, co-cricket)-free graphs is quadratically $\chi$-bounded.

\begin{theorem}\label{thm:fork-diamondK1-free}
Let $G$ be a  (fork, co-cricket)-free graph. Then either $G$ is claw-free or $G$ admits a homogeneous set decomposition or for each vertex $u$ in $G$, $G[M(u)]$ is perfect.
\end{theorem}

\begin{proof}
Let $G$ be a (fork, co-cricket)-free graph. Suppose to the contrary that none of the assertions hold. We may assume that $G$ is connected.  Let $x$ be a vertex in $G$ such that $G[M(x)]$ is not perfect. Since $G$ is co-cricket-free, we see that  $G[M(x)]$ does not contain a diamond, and hence does not contain an  odd antihole except $\overline{C_5}$. So by Theorem~\ref{thm:spgt}, $G[M(x)]$ contains an odd hole. Let $C:=$ $v_1$-$v_2$-$\cdots$-$v_{\ell}$-$v_1$ be a shortest odd hole in $G[M(x)]$ for some $\ell\ge 5$, and let $S$ denote the vertex set of $C$.

\begin{claim}\label{clm:aC}
If $v$ is a vertex in $G$ which has three consecutive neighbors in $C$, then $v$ is complete to $S\cup M(S)$.
\end{claim}
We may assume that $v$ is adjacent to the vertices $v_1$, $v_2$ and $v_3$.  Now if there is a vertex in $M(S)$ that is nonadjacent to $v$, say $a$, then $\{v_1,v_2,v_3,v, a\}$ induces a co-cricket, a contradiction. So $v$ is complete to $M(S)$. In particular, $v$ is adjacent to $x$. Next suppose to the contrary that $v$ is not complete to $C$.  Let $k \in \{4,5,\ldots, \ell\}$ be the least positive integer such that $v$ is adjacent to $v_{k-1}$, and $v$ is nonadjacent to $v_k$. Now if $k\neq \ell$, then $\{x,v,v_{k-1},v_k,v_1\}$ induces a fork, and if $k= \ell$, then $\{v_{\ell},v_{\ell-1},v,x, v_3\}$ induces a fork, a contradiction. So $v$ is complete to $S$. This proves \ref{clm:aC}.

\begin{claim}\label{clm:Cneigh}
Let $v$ be a vertex in $G$ which has a  neighbor in $C$. Then the following hold:
\begin{enumerate}
 \item[(a)] If  $\ell =5$, then  $N(v)\cap S$ is either $\{v_j, v_{j+1}\}$ or  $\{v_j, v_{j+1}, v_{j+3}\}$ or $S$, for some $j\in \{1,2,\ldots,5\}, j$ mod $5$.
 \item[(b)] If  $\ell \ge 7$, then  $N(v)\cap S$ is either $\{v_j, v_{j+1}\}$ or  $S$, for some $j\in \{1,2,\ldots, \ell\}$, $j$ mod $\ell$.
     \end{enumerate}
\end{claim}
If $v$ has three consecutive vertices of $C$ as neighbors, then by \ref{clm:aC}, $N(v)\cap S = S$, and we conclude the proof. So we may assume that no three consecutive vertices of $C$ are neighbors of $v$.

Now suppose that \ref{clm:Cneigh}:$(a)$ does not hold.  So by our assumption, there is an index $j\in \{1,2,\ldots,5\}$, $j$ mod $5$ such that $v$ is adjacent to $v_j$, and $v$ is anticomplete to $\{v_{j+1},v_{j-1},v_{j-2}\}$. Then $\{v_{j-2},v_{j-1},v_j,v_{j+1},v\}$ induces a fork, a contradiction. So \ref{clm:Cneigh}:$(a)$ holds.

Next suppose that \ref{clm:Cneigh}:$(b)$ does not hold. First let us assume that no two consecutive vertices of $C$ are neighbors of $v$. Since $v$ has a neighbor in $C$, we may assume that $v$ is adjacent to $v_1$. By assumption, $v$ is not adjacent to both $v_2$ and $v_{\ell}$. Then since $\{v_3,v_2,v_1,v_{\ell},v\}$ does not induce a fork, $v$ is adjacent to $v_3$. Likewise, $v$ is adjacent to $v_{\ell -1}$. Also by our assumption, $v$ is not adjacent to $v_4$. Now $\{v_4,v_3,v,v_1,v_{\ell-1}\}$ induces a fork, a contradiction. So we may assume that there is an index $j\in \{1,2,\ldots, \ell\}$, $j$ mod $\ell$ such that $v$ is adjacent to both $v_j$ and $v_{j+1}$, say $j=\ell$.  Moreover, by our contrary assumption, $v$ has a neighbor in $\{v_3,v_4,\ldots, v_{\ell-2}\}$. Also by our earlier arguments, $v$ is nonadjacent to both $v_2$ and $v_{\ell-1}$.  Suppose that $v$ has a neighbor in $\{v_3,v_4,\ldots, v_{\lceil\frac{\ell}{2}\rceil-1}\}$; let $t$ be the least possible integer in $\{3, 4,\ldots, \lceil\frac{\ell}{2}\rceil-1\}$ such that $v$ is adjacent to $v_t$. Now since $\{v_{\ell-1},v_{\ell},v,x,v_t\}$ does not induce a fork, $v\in M(x)$. Then since $v$-$v_1$-$v_2$-$\cdots$-$v_t$-$v$ is not an odd hole in $G[M(x)]$ which is shorter than $C$, we see that $t$ is odd.
Also since $\{v_{\ell},v,v_t,v_{t-1},v_{t+1}\}$ does not induce a fork, $v$ is adjacent to $v_{t+1}$. So by our assumption, $v$ is nonadjacent to $v_{t+2}$. Moreover, since $v$-$v_{t+1}$-$v_{t+2}$-$\cdots$-$v_{\ell-1}$-$v_{\ell}$-$v$ is not an odd hole in $G[M(x)]$ which is shorter than $C$, $v$ has a neighbor in $\{v_{t+3},v_{t+4},\ldots, v_{\ell-2}\}$, say $v_k$. But now $\{v_2,v_1,v,v_{t+1},v_k\}$ induces a fork, a contradiction. Thus, by using symmetry, we may assume that $v$ has no neighbor in $\{v_2,v_3,\ldots, v_{\lceil\frac{\ell}{2}\rceil-1}, v_{\lceil\frac{\ell}{2}\rceil+1},\ldots, v_{\ell -1}\}$. So by our contrary assumption, $v$ is adjacent to  $v_{\lceil\frac{\ell}{2}\rceil}$. But then $\{v_1,v,v_{\lceil\frac{\ell}{2}\rceil}, v_{\lceil\frac{\ell}{2}\rceil-1}, v_{\lceil\frac{\ell}{2}\rceil+1}\}$ induces a fork, a contradiction. So \ref{clm:Cneigh}:(b) holds.  This  proves \ref{clm:Cneigh}.

\medskip
Let $X$ be the set $\{v\in V(G)\setminus S\mid N(v)\cap S=S\}$, and let $Y$ be the set $\{v\in V(G)\setminus S \mid N(v)\cap S = \{v_j, v_{j+1}\}, \mbox{ for some }  j\in \{1,2,\ldots,\ell\},  j  \mod  \ell  \}$. Moreover, if $\ell=5$, then let $Z$ be the set
$\{v\in V(G)\setminus S \mid N(v)\cap S = \{v_j, v_{j+1}, v_{j+3}\}, \mbox{ for some }  j\in \{1,2,\ldots,\ell\},  j  \mod  \ell\}$, otherwise let $Z=\emptyset$. Then by \ref{clm:Cneigh}, we immediately have the following assertion.

\begin{claim}\label{clm:neiS}
$N(S) = X\cup Y\cup Z$, and so $V(G)= S\cup X\cup Y\cup Z\cup M(S)$.
\end{claim}

\begin{claim}\label{clm:Xemp}
$X = \emptyset$.
\end{claim}
 Suppose to the contrary that $X$ is nonempty. We claim that $X$ is complete to $V(G)\setminus X = S\cup Y\cup Z\cup M(S)$.  Suppose there are nonadjacent vertices, say $p\in X$ and $q\in V(G)\setminus X$. Recall that, by \ref{clm:aC}, $X$ is complete to $S\cup M(S)$; in particular, $x$ is complete to $X$. So $q\in Y\cup Z$.  Then by our definitions of $Y$ and $Z$, there is an index $j\in \{1,2,\ldots,\ell\},  j  \mod  \ell$ such that $q$ is complete to $\{v_{j},v_{j+1}\}$, and anticomplete to $\{v_{j-1},v_{j+2}\}$, say $j=1$. Now if $q$ is adjacent to $x$, then $\{q,x,p,v_3,v_{\ell}\}$ induces a fork, and if $q$ is nonadjacent to $x$, then $\{q,v_1,p,x,v_3\}$ induces a fork. These contradictions show that $X$ is complete to $Y\cup Z$. Thus   $X$ is complete to $V(G)\setminus X$. But then  since $S\subseteq V(G)\setminus X$, we see that $V(G)\setminus X$ is a homogeneous set in $G$, which is a contradiction. This proves \ref{clm:Xemp}.

\begin{claim}\label{clm:ZnonneighC}
$Z$ is anticomplete to $M(S)$.
\end{claim}
Suppose to the contrary that there are adjacent vertices, say  $p\in Z$ and $q\in M(S)$. Since $p\in Z$, by the definition of $Z$, we may assume that    $\ell=5$, and   there is an index $j\in  \{1,2,\ldots,5\},  j  \mod  5$  such that $N(p)\cap S = \{v_j, v_{j+1}, v_{j+3}\}$, say $j=1$.
But now $\{q,p,v_4,v_5,v_3\}$ induces a fork, a contradiction.  This proves \ref{clm:ZnonneighC}.

\begin{claim}\label{clm:neiCZ}
If $Z \not =  \emptyset$, then $N(S)=Z$.
\end{claim}
 Let $p\in Z$. Then by our definition of $Z$, we may assume that $\ell=5$, and so there is an index $j\in  \{1,2,\ldots,5\},  j  \mod  5$  such that $N(p)\cap S = \{v_j, v_{j+1}, v_{j+3}\}$, say $j=1$. Moreover, by \ref{clm:ZnonneighC}, $p$ is nonadjacent to $x$.  Recall that  $N(S)=X\cup Y\cup Z$, and by \ref{clm:Xemp}, we know that $X=\emptyset$. So we show that $Y=\emptyset$. Suppose to the contrary that $Y$ is nonempty, and let $q\in Y$. Then by the definition of $Y$, there is an index $k\in  \{1,2,\ldots,5\},  k  \mod  5$  such that $N(q)\cap S = \{v_k, v_{k+1}\}$. Then, up to symmetry, we have three cases:

 \no $\bullet$ $k=1$:    If $p$ is adjacent to $q$, then $\{q,p,v_4,v_5,v_3\}$ induces a fork, a contradiction, and so $p$ is nonadjacent to $q$. Then since $\{p,v_1,q,v_2,x\}$ does not induce a co-cricket, $q$ is adjacent to $x$. But then $\{x,q,v_1,v_5,p\}$ induces a fork, a contradiction.

 \no $\bullet$ $k=2$: If $p$ is nonadjacent to $q$, then $\{q,v_3,v_4,v_5,p\}$ induces a fork, a contradiction, and so $p$ is  adjacent to $q$. Then since $\{p,q,v_1,v_2,x\}$ does not induce a co-cricket, $q$ is adjacent to $x$. But then $\{x,q,p,v_1,v_4\}$ induces a fork, a contradiction.

 \no $\bullet$ $k=3$: If $p$ is nonadjacent to $q$, then $\{v_5,v_4,p,v_2,q\}$ induces a fork, a contradiction, and so $p$ is  adjacent to $q$. Then since $\{p,q,v_3,v_4,x\}$ does not induce a co-cricket, $q$ is adjacent to $x$. But then $\{v_1,p,q,v_3,x\}$ induces a fork,  a contradiction.

The above contradictions show that $Y$ is empty. This proves \ref{clm:neiCZ}.

 \begin{claim}\label{clm:Zemp}
 $Z=\emptyset$.
\end{claim}
Suppose to the contrary that $Z$ is nonempty.  So $\ell=5$. Moreover, by \ref{clm:neiCZ}, we have $N(S)=Z$, and by \ref{clm:ZnonneighC}, $Z$ is anticomplete to $M(S)$. But then, since $x\in M(S)$, we conclude that the graph is not connected, a contradiction.  This proves \ref{clm:Zemp}.

\medskip
Now by \ref{clm:neiS}, \ref{clm:Xemp} and \ref{clm:Zemp}, we conclude that $N(S) = Y$, and hence we have the following.

\begin{claim}\label{mainclm}
If $v$ is a vertex in $G$ which has a  neighbor  in $C$, then there is an index $j\in \{1,2,\ldots, \ell\}$, $j$ mod $\ell$ such that $N(v)\cap S= \{v_j, v_{j+1}\}$.
\end{claim}

Now let $K$ be an induced
claw with vertex set $\{a,b, c, d\}$ and edge set $\{ab,ac,$ $ad\}$. By \ref{mainclm}, $K$
cannot have more than two vertices on $C$. Also, at most one vertex in $\{b,c,d\}$ belongs to $C$. Then, up to symmetry, we have the following cases.

\smallskip
\no $\bullet$ $V(K)\cap S= \{a,d\}$: Let $a=v_1$ and $d=v_{\ell}$. Then by \ref{mainclm}, $bv_2,cv_2\in E$. But then again by \ref{mainclm}, $\{v_{\ell-1},d,a,b,c\}$ induces a fork, a contradiction.

\smallskip
\no $\bullet$ $V(K)\cap S= \{a\}$: Let $a=v_1$. Then by \ref{mainclm}, at least two vertices in $\{b,c,d\}$ are adjacent to either $v_2$ or $v_{\ell}$, say $b$ and $c$ are adjacent to $v_2$. Then again by \ref{mainclm}, $\{v_4,v_3,v_2, b,c\}$ induces a fork, a contradiction.

\smallskip
\no $\bullet$ $V(K)\cap S= \{d\}$:
Let $d=v_1$. Then by \ref{mainclm}, we may assume that $av_2\in E$. Then by \ref{mainclm}, since $\{v_{\ell},d,a,b,c\}$ does not induce a fork, $v_{\ell}$ has a neighbor in  $\{b,c\}$.  Also, to avoid an induced claw with center in $C$, $v_{\ell}$ has a nonneighbor in  $\{b,c\}$. So we may assume that  $v_{\ell}$ is adjacent to $b$, and nonadjacent to $c$. Thus by \ref{mainclm}, $N(b)\cap S =\{v_{\ell-1}, v_{\ell}\}$. Then since $\{v_3, v_2,a,b,c\}$ does not induce a fork, by \ref{mainclm}, $c$ is adjacent to $v_3$. But then, by \ref{mainclm}, $\{v_3,c, a, b,d\}$ induces a fork which is a contradiction.

\smallskip
 \no $\bullet$ \emph{$V(K)\cap S= \es$ and $a$ has a  neighbor on $C$}: By \ref{mainclm}, we may assume that $N(a)\cap S =\{v_1,v_2\}$. To avoid an induced claw intersecting $C$, both $v_1$ and $v_2$ have exactly two neighbors among $b,c,d$, and thus we may assume that $v_1$ is adjacent to $b$ and $c$, and not adjacent to $d$. Again  to avoid an induced claw intersecting $C$, exactly one of $b,c$ is adjacent to $v_{\ell}$, say $b$ is adjacent to $v_{\ell}$, and so by \ref{mainclm}, $N(b)\cap S = \{v_1, v_{\ell}\}$. Moreover, by \ref{mainclm}, $N(c)\cap S = \{v_1, v_2\}$. Then since $\{v_3, v_2, a,b,d\}$ does not induce a fork, by \ref{mainclm}, $d$ is adjacent to $v_3$. But then $\{v_3,d,a,b,c\}$ induces a fork, a contradiction.

\smallskip
 \no $\bullet$ \emph{$V(K)\cap S= \es$ and $b$ has a  neighbor on $C$}: By \ref{mainclm}, we may assume that $N(b)\cap S =\{v_1,v_2\}$, and we may assume that $a$ has no neighbors on $C$.  Then since $\{v_1,b,a,c,d\}$ does not induce a fork, we may assume that $c$ is adjacent to $v_1$. Then to avoid an induced claw intersecting $C$, $c$ is adjacent to $v_{\ell}$, and so by \ref{mainclm}, $N(c)\cap S =\{v_1,v_{\ell}\}$. Then since $\{v_2,b,a,c,d\}$ does not induce a fork, $d$ is adjacent to $v_2$. So by \ref{mainclm}, $d$ is not adjacent to $v_{\ell}$. But then $\{v_{\ell},c,a,b,d\}$ induces a fork,  a contradiction.

\smallskip
 \no $\bullet$ \emph{$V(K)\cap S= \es$ and no vertex of $K$ has a neighbor on $C$}: Since $G$ is connected,
there exists a $j\in \{1,2,\ldots,\ell\}$ and a shortest path $p_1$-$p_2$-$\cdot$-$p_t$-$a$, say $P$, such that $t\ge 2$, $v_j=p_1$,  and $p_2$ has a neighbor on $C$. By the choice of $P$, no vertex of this path has a neighbor on $C$ except $p_2$. We may assume that $j=1$. Then by  \ref{mainclm}, we may further assume that $N(p_2)\cap S = \{v_1,v_2\}$. If $p_t=b$, then $\{v_{\ell},v_1(=p_1),p_2,\ldots,p_t(=b), a,c,d\}$ induces a fork, a contradiction. So we may assume that $p_t\notin\{b,c,d\}$. Now if $p_t$ has two or more neighbors in $\{b,c,d\}$, then $\{a,b,c,d,p_t,v_{\ell}\}$ induces a graph containing a  co-cricket, a contradiction, and if $p_t$ has  exactly one neighbor in $\{b,c,d\}$, say $b$ or if $p_t$ has  no neighbor in $\{b,c,d\}$, then $\{a,c,d,p_t,p_{t-1}\}$ induces a fork, a contradiction.

This completes the proof of the theorem.
\end{proof}

\begin{cor} \label{FCCfreeStruc-2}Let $G$ be a (fork,\,co-cricket)-free graph.  Then either $G$ is claw-free or $G$ is perfectly
weight divisible, and hence perfectly divisible.\end{cor}
\begin{proof}This follows from Theorems~\ref{thm:hsd} and \ref{pwd=pd}, and from Theorem~\ref{thm:fork-diamondK1-free}. \end{proof}

\begin{cor}
Every (fork,\,co-cricket)-free graph $G$ satisfies $\chi(G)\le \omega(G)^2$.
\end{cor}
\begin{proof}This follows from Corollary~\ref{FCCfreeStruc-2}, Theorem~\ref{thm:claw-free},  and Lemma~\ref{lem:pd-bound}. \end{proof}

\subsection{The class of (fork, bull)-free graphs}

In this section, we observe that (fork, bull)-free graphs are perfectly divisible, and hence  the class of (fork, bull)-free graphs is quadratically $\chi$-bounded. We use the following theorem.

\begin{theorem}[\cite{KM}]\label{lem:FBull-free}
If $G$ is a (fork,\,bull)-free graph that does not admit a homogeneous set decomposition, then for every vertex $v$ in $G$, $G[M(v)]$ is odd hole-free and $\overline{P_5}$-free, and hence perfect.
\end{theorem}

\begin{cor} \label{FBullfreeStruc-2}Let $G$ be a (fork,\,bull)-free graph.  Then   $G$ is perfectly
weight divisible, and hence perfectly divisible.\end{cor}
\begin{proof}This follows from Theorems~\ref{thm:hsd} and \ref{pwd=pd}, and from Theorem~\ref{lem:FBull-free}. \end{proof}

\begin{cor}\label{thm:forkbull}
Let $G$ be a (fork, bull)-free graph. Then $\chi(G) \leq  \binom{\omega(G) + 1}{2}$.
\end{cor}
\begin{proof}This follows from Corollary~\ref{FBullfreeStruc-2} and Lemma~\ref{lem:pd-bound}. \end{proof}

  \section{Concluding remarks and open problems}
  We have studied  the structure  of (fork,\,$F$)-free graphs in the context of perfect divisibility, where $F$ is some graph on five vertices  with a stable set of size 3, and obtained   quadratic $\chi$-binding functions for such classes of graphs. Recall that if $|V(F)|\leq 4$, then the class of (fork,\,$F$)-free graphs is known to be quadratically $\chi$-bounded except when $F=4K_1$.

\begin{problem}
What is the smallest $\chi$-binding function for the class of (fork, $4K_1$)-free graphs?
\end{problem}

Further, it will be interesting to study  $\chi$-binding functions for the class of (fork,\,$F$)-free graphs, where $F$ is a graph on five vertices with stable sets of size at most 2, in particular, for the class of (fork, $C_5$)-free graphs and for the class of (fork, $\overline{P_5}$)-free graphs.

\medskip

The notion of perfect divisibility played a key role in proving quadratic $\chi$-binding functions for some classes of fork-free graphs.  In this paper, we showed that class of (fork, $F$)-free graphs is perfectly divisible, where $F\in \{P_6, \text{co-dart}, \text{bull}\}$, and the class of  (fork, $H$)-free graphs is either claw-free or perfectly divisible, when $F\in$ \{dart, banner, co-cricket\}. Indeed, the third author conjectured the following.

\begin{conjecture}
Every fork-free graph is perfectly divisible.
\end{conjecture}

The above conjecture is not even known to be true for a very special subclass of fork-free graphs, namely claw-free graphs. It is conceivable that the proof of perfect divisibility for claw-free graphs could be based on the detailed structure theorem for such graphs due to Chudnovsky and Seymour \cite{CS}.  However, the proof of perfect divisibility for a subclass of claw-free graphs, namely, the class of line graphs seems to be easy as given below.

\begin{proposition}
Every line graph is perfectly divisible.
\end{proposition}

\begin{proof}
Let $G$ be a connected graph. If $G$ is a tree, then $L(G)$ is perfect, and hence perfectly divisible. Hence we may assume that $G$ is not a tree. Let $T$ be a spanning tree of $G$. Consider the partition of $E(G) =   E(T)\cup (E(G) \setminus E(T))$. In $L(G)$, $E(T)$ induces a perfect graph since it is the line graph of a tree. Also, the clique number of $L(G) - E(T)$ is smaller than that of $L(G)$. Hence $L(G)$ is perfectly divisible. Since perfect divisibility is preserved under disjoint union, we are done.
\end{proof}

It is well known that the complement of a perfect graph is perfect \cite{LL}. How about perfectly divisible graphs? The following proposition shows that it fails badly.

\begin{proposition}
The class of graphs whose complements are perfectly divisible is not $\chi$-bounded.
\end{proposition}

\begin{proof}
We know that the class of $3K_1$-free graphs is perfectly divisible \cite{Hoang}. Hence the class of triangle-free graphs is contained in the class of graphs whose complements are perfectly divisible. It is well known that the class of triangle-free graphs have unbounded chromatic number, establishing the result.
\end{proof}

But what about graphs $G$ such that both $G$ and its complement $\overline{G}$ are perfectly divisible?

\begin{problem}
What is the smallest $\chi$-binding function for  the class of graphs $\cal G$ such that for each $G\in \cal G$,  both $G$ and  $\overline{G}$ are perfectly divisible?
\end{problem}

 Not much is known about the class of perfectly divisible graphs in general.  Perhaps determining graphs that are not in the class and minimal with respect to that property will shed light on this.

\begin{problem}
Determine the set of forbidden induced subgraphs for the class of perfectly divisible graphs.
\end{problem}

%It is to be noted that the corresponding problem for perfect graphs was notoriously difficult and solved only after 40 years of intense effort \cite{spgt}. But there at least there was a conjectured answered the forbidden set was extremely simple. In fact, it is not even known how to construct  graphs that are not perfectly divisible. The only way we know is to construct graphs which fail the quadratic $\chi$-binding function that all perfectly divisible graphs have to satisfy. For example, applying the well-known Mycielski's construction to a cycle of odd length greater than $3$ always gives a graph that is not perfectly divisible.

\medskip

\no{\bf Acknowledgement}
The authors would like to thank Maria Chudnovsky for participating in the initial stages of this paper.

\end{document}